\documentclass[sn-mathphys]{sn-jnl}

\usepackage{amssymb,amsmath,mathrsfs, amsthm}

\jyear{2021}%

\theoremstyle{thmstyleone}%
\numberwithin{equation}{section}
\newtheorem{theorem}{Theorem}[section]
\newtheorem{proposition}[theorem]{Proposition}%
\newtheorem{corollary}{Corollary}[section]
\theoremstyle{thmstyletwo}%
\newtheorem{example}{Example}%
\newtheorem{remark}[theorem]{Remark}%
\newtheorem{lemma}{Lemma}[section]

\theoremstyle{thmstylethree}%
\newtheorem{definition}[theorem]{Definition}%

\begin{document}

\title[Article Title]{Schwarz lemma on bounded symmetric domains endowed with holomorphic invariant K\"ahler--Berwald metrics\\
}

\author{\fnm{Yong} \sur{He$^1$}}\email{heyong@xjnu.edu.cn}
\author{\fnm{Chunping} \sur{Zhong$^2$}\footnote{Corresponding author}}\email{zcp@xmu.edu.cn}


\affil[1]{\orgdiv{School of Mathematical Sciences}, \orgname{Xinjiang Normal University},\city{Urumqi}, \postcode{830017}, \country{China}}

\affil[2]{\orgdiv{School of Mathematical Sciences}, \orgname{Xiamen University}, \city{Xiamen}, \postcode{361005}, \country{China}}



\abstract{
We prove  a rigidity theorem: a globally symmetric complex Finsler space $(M, J, F)$ is necessarily a K\"ahler-Berwald space, namely $F$ must be a K\"ahler-Berwald metric. We  also obtain a Schwarz lemma for holomorphic mappings $f$ from an arbitrary bounded symmetric domain $\mathfrak{D}$ into itself whenever $\mathfrak{D}$ is endowed with an $\mbox{Aut}(\mathfrak{D})$-invariant K\"ahler-Berwald metric $F$ such that its holomorphic sectional curvature is bounded below and above by negative constants $-K_1<0$ and $-K_2<0$, respectively. The novelty of this Schwarz lemma is that the Lu constant of $(\mathfrak{D},F)$ is optimal for each given $F$ whenever rank$(\mathfrak{D})\geq 2$.

}




\maketitle

\section{Introduction and statement of the main theorem}

A \emph{complex Finsler manifold} is a triple \((M,J,F)\), where \(M\) is a (connected) complex manifold, \(J\) is a complex structure on \(M\), and \(F\colon T^{1,0}M\to[0,+\infty)\) is a strongly convex complex Finsler metric in the sense of Abate and Patrizio \cite{AP}. More precisely, for
\[
y_o:=\frac{1}{2}(y-iJy)\in T^{1,0}M,\qquad y\in TM,
\]
the function
\[
F^\circ(y):=F(y_o)
\]
is a real Finsler metric on \(M\) (regarded as a smooth manifold). It is easy to check that
\[
F^\circ(Jy)=F^\circ(y),\qquad \forall\, y\in TM.
\]
A complex Finsler manifold \((M,J,F)\) is called \emph{homogeneous} if the group \(A(M,J,F)\) of holomorphic isometries of \((M,J,F)\) acts transitively on \(M\). Deng \cite{Deng} proved that a homogeneous complex Finsler manifold \((M,J,F)\) can be written as a coset space \(G/H\), where \(G=A_0(M,J,F)\) is the identity component of \(A(M,J,F)\) and \(H\) is the isotropy subgroup of \(G\) at a fixed point \(o\in M\).

A complex Finsler space \((M,J,F)\) is called \emph{globally symmetric} if for every point \(p\in M\), there exists an involutive holomorphic isometry \(\sigma_p\) of \((M,J,F)\) such that \(p\) is an isolated fixed point of \(\sigma_p\) \cite{Deng}. Clearly, every globally symmetric complex Finsler space is homogeneous.

Deng \cite{Deng} also proved that every globally symmetric complex Finsler space \((G/H,J,F)\) admits a \(G\)-invariant Riemannian metric \(Q\) on \(G/H\) such that \((G/H,J,Q)\) is a Hermitian symmetric space. Hermitian symmetric spaces were completely classified by Cartan \cite{Cartan} (see also Helgason \cite{Helgason}) as follows:

\begin{table}[htbp]
\centering
\caption{Irreducible Hermitian symmetric spaces ($p,q\geq 1$)}
\resizebox{\textwidth}{!}{%
\begin{tabular}{lllll}
\hline
 & Noncompact spaces & Compact spaces & Rank & Dimension \\
\hline
AIII
  & \(SU(p,q)/S(U_p\times U_q)\)
  & \(SU(p+q)/S(U_p\times U_q)\)
  & \(\min(p,q)\)
  & \(2pq\) \\
BDI
  & \(SO_0(p,2)/(SO(p)\times SO(2))\)
  & \(SO(p+2)/(SO(p)\times SO(2))\)
  & \(2\)
  & \(2p\) \\
DIII
  & \(SO^*(2n)/U(n)\quad (n\ge 4)\)
  & \(SO(2n)/U(n)\quad (n\ge 4)\)
  & \([\frac{n}{2}]\)
  & \(n(n-1)\) \\
CI
  & \(Sp(n,\mathbb R)/U(n)\quad (n\ge 2)\)
  & \(Sp(n)/U(n)\quad (n\ge 2)\)
  & \(n\)
  & \(n(n+1)\) \\
EIII
  & \((\mathfrak{E}_{6(-14)},\mathfrak{so}(10)+\mathbb R)\)
  & \((\mathfrak{E}_{6(-78)},\mathfrak{so}(10)+\mathbb R)\)
  & \(2\)
  & \(32\) \\
EVIII
  & \((\mathfrak{E}_{7(-25)},\mathfrak{e}_6+\mathbb R)\)
  & \((\mathfrak{E}_{7(-133)},\mathfrak{e}_6+\mathbb R)\)
  & \(3\)
  & \(54\) \\
\hline
\end{tabular}
}
\end{table}

On the other hand, Szab\'o \cite{Szabo} proved that every locally irreducible globally symmetric Riemannian manifold \(G/H\) of rank \(\geq 2\) admits infinitely many \(G\)-invariant real Finsler metrics (see Theorems 2 and 8 in \cite{Szabo}), whereas no rank-one symmetric Riemannian manifold \(G/H\) admits a non-Riemannian \(G\)-invariant real Finsler metric. He also provided a list of all locally irreducible globally symmetric non-Riemannian Berwald spaces. Consequently, the classification problem for complex Finsler symmetric spaces reduces to determining which Hermitian symmetric spaces in \textbf{Table 1} admit a \(G\)-invariant non-Riemannian complex Finsler metric \(F\).

In \cite{Zhong-b}, Zhong explicitly constructed infinitely many \(\operatorname{Aut}(P_n)\)-invariant, non-Hermitian, quadratic, strongly convex complex Finsler metrics on the unit polydisk \(P_n\) in \(\mathbb{C}^n\) \((n\ge 2)\). For irreducible bounded symmetric domains, Ge and Zhong \cite{GZ} explicitly constructed infinitely many \(\operatorname{Aut}(\mathfrak{D})\)-invariant, non-Hermitian, strongly pseudoconvex complex Finsler metrics on each classical domain \(\mathfrak{D}\) with \(\operatorname{rank}(\mathfrak{D})\ge 2\). Cao, Ge and Zhong \cite{CGZ} explicitly constructed infinitely many invariant strongly pseudoconvex complex Finsler metrics on complex Grassmannian manifolds of rank \(\ge 2\). These constructions yield numerous nontrivial examples of complex Finsler symmetric spaces in both the noncompact and compact cases, as well as in the reducible and irreducible cases. It is worth noting that all the metrics constructed explicitly in these works are non-Hermitian quadratic and special K\"ahler--Finsler metrics, i.e., K\"ahler--Berwald metrics in the sense of Aikou \cite{Aikou-a, Aikou-b}.

It is well known that Hermitian symmetric spaces play an important role in both several complex variables and complex geometry; see, for example, the fundamental works of L. K. Hua \cite{Hua}, Q. K. Lu \cite{Lu-b}, and N. Mok \cite{Mok}, among many others. The explicit constructions of nontrivial (non-Hermitian) examples of complex Finsler symmetric spaces show that there is good reason to develop the theory of non-Hermitian complex Finsler symmetric spaces.

Note that the following proposition is well known for Hermitian symmetric spaces (see Proposition 4.1 in \cite{Helgason}).

\begin{proposition}[\cite{Helgason}]\label{Helgason}
Let \((M, J, h)\) be a Hermitian symmetric space. Then \(h\) is necessarily a K\"ahler metric.
\end{proposition}

\textbf{Question 1.}\quad For a non-Hermitian complex Finsler symmetric space \((M, J, F)\) (that is, \(F\) is not Hermitian), is \(F\) necessarily a K\"ahler--Berwald metric?

The following theorem answers Question 1 affirmatively.

\begin{theorem}\label{thm-a}
Let \((M, J, F)\) be a globally symmetric complex Finsler space. Then \(F\) is necessarily a K\"ahler--Berwald metric.
\end{theorem}

\begin{remark}
The above theorem was also obtained by Ge and Zhou \cite{GeZhou}; our proof of Theorem \ref{thm-a} differs from theirs.
\end{remark}

\begin{remark}
If \(M\) is a polydisk, then Theorem \ref{thm-a} is due to Lin and Zhong \cite{LZ}; if \(M\) is a polyball, then Theorem \ref{thm-a} is due to Lin, Wang, and Zhong \cite{LWZ}; if \(M\) is a complex Grassmannian manifold, then Theorem \ref{thm-a} is due to Cao, Ge, and Zhong \cite{CGZ}; and if \(M\) is any irreducible bounded symmetric domain of types I--IV, then Theorem \ref{thm-a} is due to Zhong \cite{Zhong-c}.
\end{remark}

In \cite{Lu-a}, Lu systematically investigated the Schwarz lemma and analytic invariants of a homogeneous domain \(\mathfrak{D}\subset\mathbb{C}^n\) with respect to the Bergman metric \(ds^2\) on \(\mathfrak{D}\). In particular, he obtained the following fundamental theorem.

\begin{theorem}[\cite{Lu-a}]
Let \(\mathfrak{D}\) be a bounded schlicht transitive domain, and let \(w=f(z)\) be any analytic mapping carrying \(\mathfrak{D}\) into itself.
\begin{enumerate}
\item[(1)] There exists a positive constant \(k\), depending only on \(\mathfrak{D}\), such that
\begin{equation}
ds^2(w,\overline{w})\leq k^2ds^2(z,\overline{z}).\label{lc}
\end{equation}
\item[(2)] If there is a point \(z\in\mathfrak{D}\) such that
\begin{equation}
ds^2(w,\overline{w})=ds^2(z,\overline{z}),\label{rigidity}
\end{equation}
then the mapping \(w=f(z)\) must map \(\mathfrak{D}\) topologically onto itself.
\end{enumerate}
\end{theorem}

The least positive constant \(k\), depending only on \(\mathfrak{D}\), for which \eqref{lc} holds is called the Schwarz constant of \(\mathfrak{D}\) and is denoted by \(k_0(\mathfrak{D})\).

In particular, it follows from Theorems 5, 6, 9, 11, 12, and 16 in Lu \cite{Lu-a} that

\begin{theorem}[\cite{Lu-a}]\label{Lu-a}
Let \(\mathfrak{D}\) be either the unit polydisk \(P_n\) or an irreducible bounded symmetric domain \(\mathfrak{R}_A\), where \(A=\mathrm{I},\mathrm{II},\mathrm{III},\mathrm{IV}\), and let \(ds_{\mathfrak{D}}^2\) be the Bergman metric on \(\mathfrak{D}\). Then for every holomorphic mapping \(f\) from \(\mathfrak{D}\) into itself,
\begin{equation}
f^\ast ds_{\mathfrak{D}}^2\leq k_0^2(\mathfrak{D})ds_{\mathfrak{D}}^2,\label{Lu}
\end{equation}
where
\begin{equation}
k_0(\mathfrak{D})=\sqrt{\operatorname{rank}(\mathfrak{D})}=
\begin{cases}
\sqrt{r}, & \mathfrak{D}=\mathfrak{R}_I(r,s),\\[2mm]
\sqrt{p}, & \mathfrak{D}=\mathfrak{R}_{II}(p),\\[2mm]
\sqrt{\left[\frac{q}{2}\right]}, & \mathfrak{D}=\mathfrak{R}_{III}(q),\\[2mm]
\sqrt{2}, & \mathfrak{D}=\mathfrak{R}_{IV}(n).
\end{cases}
\end{equation}
Hence the constant \(k_0(\mathfrak{D},ds_{\mathfrak{D}}^2)\) is an analytic invariant of \(\mathfrak{D}\).
\end{theorem}

In this paper, we call \(k_0(\mathfrak{D})\) the Lu constant of \(\mathfrak{D}\) with respect to the Bergman metric \(ds^2\), and denote it by \(k_0(\mathfrak{D},ds^2)\).

A natural question is whether Theorem \ref{Lu-a} can be generalized to the case where \(\mathfrak{D}\) admits an \(\operatorname{Aut}(\mathfrak{D})\)-invariant strongly pseudoconvex complex Finsler metric that is not Hermitian quadratic. It is well known that most intrinsic metrics constructed in geometric function theory in several complex variables lack good smoothness, such as the Carath\'eodory and Kobayashi metrics on the unit polydisk \(P_n\) or on bounded symmetric domains of rank \(\ge 2\). Before 2023, the only known holomorphically invariant Hermitian metric was the Bergman metric, introduced by S. Bergman in 1922 and now bearing his name; it was also the first K\"ahler metric to be studied. The first example of a holomorphically invariant non-Hermitian strongly pseudoconvex complex Finsler metric was explicitly constructed on the unit polydisk \(P_n\subset\mathbb{C}^n\) \((n\ge 2)\) by Zhong \cite{Zhong-b}. In \cite{GZ}, Ge and Zhong explicitly constructed holomorphically invariant non-Hermitian strongly pseudoconvex complex Finsler metrics on irreducible bounded symmetric domains \(\mathfrak{D}\) of types I--IV with \(\operatorname{rank}(\mathfrak{D})\ge 2\).

In \cite{Zhong-c}, Zhong obtained the following theorem for classical domains (irreducible bounded symmetric domains of types I--IV).

\begin{theorem}[\cite{Zhong-c}]
Let \(\mathfrak{D}\) be a classical domain endowed with an \(\operatorname{Aut}(\mathfrak{D})\)-invariant K\"ahler--Berwald metric \(F\) whose holomorphic sectional curvature is bounded from below and from above by negative constants \(-K_1\) and \(-K_2\), respectively. Then for every holomorphic mapping \(f\) from \(\mathfrak{D}\) into itself, we have
\begin{equation}
(f^\ast F)(Z;V)\leq \sqrt{\frac{K_1}{K_2}}F(Z;V),\quad \forall\,(Z;V)\in T^{1,0}\mathfrak{D}.\label{sc}
\end{equation}
\end{theorem}

\textbf{Question 2.} Can one establish a Schwarz lemma on an arbitrary bounded symmetric domain \(\mathfrak{D}\) endowed with an arbitrary \(\operatorname{Aut}(\mathfrak{D})\)-invariant K\"ahler--Berwald metric?

The following theorem gives an affirmative answer.

\begin{theorem}\label{thm-b}
Let \(\mathfrak{D}\subset\mathbb{C}^n\) be a bounded symmetric domain in its Harish-Chandra realization, and let \(F:T^{1,0}\mathfrak{D}\to[0,+\infty)\) be an \(\operatorname{Aut}(\mathfrak{D})\)-invariant K\"ahler--Berwald metric whose holomorphic sectional curvature is bounded from below and from above by negative constants \(-K_1<0\) and \(-K_2<0\), respectively. Then
\begin{enumerate}
\item[(1)] for every holomorphic mapping \(f:\mathfrak{D}\to \mathfrak{D}\) and every \((z;v)\in T^{1,0} \mathfrak{D}\), we have
\begin{equation}
(f^*F)(z;v) \le \sqrt{\frac{K_1}{K_2}}\, F(z;v);
\label{eq-2}
\end{equation}
\item[(2)] the constant \(\sqrt{\frac{K_1}{K_2}}\) is optimal in the sense that there exist a point \(z_0\in\mathfrak{D}\), a nonzero vector \(v_0\in T_{z_0}^{1,0}\mathfrak{D}\), and a holomorphic self-map \(f_0:\mathfrak{D}\to\mathfrak{D}\) such that equality holds in \eqref{eq-2} with \(f=f_0\) at \((z_0;v_0)\);
\item[(3)] if, moreover, there exists a point \(z_0\in \mathfrak{D}\) such that
\begin{equation}
(f^*F)(z_0;v)=F(z_0;v) \quad \text{for all nonzero } v\in T_{z_0}^{1,0}\mathfrak{D},
\label{thm-s2}
\end{equation}
then \(f\in\operatorname{Aut}(\mathfrak{D})\).
\end{enumerate}
\end{theorem}

\begin{remark}
 If \(\mathfrak{D}\) is the unit polydisk or one of the irreducible bounded symmetric domains of types I--IV, and \(F\) is the Bergman metric on \(\mathfrak{D}\), then
\[
\sqrt{\frac{K_1}{K_2}}=\sqrt{\operatorname{rank}(\mathfrak{D})},
\]
and Theorem \ref{thm-b} is due to Look \cite{Lu-a}. If \(\mathfrak{D}\) is one of the irreducible bounded symmetric domains of types I--IV and \(F\) is an arbitrary \(\operatorname{Aut}(\mathfrak{D})\)-invariant K\"ahler--Berwald metric (which need not be the Bergman metric on \(\mathfrak{D}\) whenever \(\operatorname{rank}(\mathfrak{D})\ge 2\)), then the assertions (1) and (2) of Theorem \ref{thm-b} are due to Zhong \cite{Zhong-c}. If \(\mathfrak{D}=P_n\) is the unit polydisk in \(\mathbb{C}^n\) endowed with an arbitrary \(\operatorname{Aut}(P_n)\)-invariant K\"ahler--Berwald metric \(F\), then the assertions (1) and (2) of Theorem \ref{thm-b} are due to Lin and Zhong \cite{LZ}. If \(\mathfrak{D}\) is a polyball endowed with an arbitrary holomorphic invariant K\"ahler--Berwald metric \(F\), then the assertions (1) and (2) of Theorem \ref{thm-b} are due to Lin, Wang, and Zhong \cite{LWZ}.
\end{remark}

\begin{remark}
The novelty of Theorem \ref{thm-b} lies in the following three respects: (i) \(\mathfrak{D}\) is an arbitrary bounded symmetric domain; (ii) the Lu constant \(k_0(\mathfrak{D},F)\) is optimal; and (iii) the rigidity property of the Bergman metric in \eqref{rigidity} is also shared by every non-Hermitian \(\operatorname{Aut}(\mathfrak{D})\)-invariant K\"ahler--Berwald metric \(F\) whose holomorphic curvature is bounded from below and from above by the negative constants \(-K_1<0\) and \(-K_2<0\), respectively.
\end{remark}

\section{Preliminaries}

\begin{definition}[\cite{AP}]
An upper semicontinuous complex Finsler metric $F$ on a complex manifold $M$ is an upper semicontinuous function $F:T^{1,0}M\rightarrow[0,+\infty)$ satisfying
\begin{enumerate}
\item[(i)] $F(p;v)>0$ for all $p\in M$ and $v\in T_p^{1,0}M$ with $v\neq 0$;
\item[(ii)] $F(p;\lambda v)=\vert\lambda\vert F(p;v)$ for all $p\in M, v\in T_p^{1,0}M$ and $\lambda \in\mathbb{C}$.
\end{enumerate}
\end{definition}

In the following, we denote by $\mathbb{D}=\{z\in\mathbb{C}\vert \vert z\vert<1\}$ the open unit disk in $\mathbb{C}$.
\begin{definition}[\cite{AP}]
A pseudohermitian metric $\mu_g$ of scale $g$ on $\mathbb{D}$ is the upper semicontinuous pseudometric on $T^{1,0}\mathbb{D}$ defined by
\begin{equation}
\mu_g=gd\zeta\otimes d\overline{\zeta},\label{ph}
\end{equation}
where $g:\mathbb{D}\rightarrow[0,+\infty)$ is a non-negative upper semicontinuous function such that the set $S_g=g^{-1}(0)$ is a discrete subset of $\mathbb{D}$.
\end{definition}

\begin{proposition}\label{prop-iph}
Let $F:T^{1,0}M\rightarrow[0,+\infty)$ be an upper semicontinuous complex Finsler metric on a complex manifold of dimension $n$. Then for each non-constant holomorphic map $\varphi:\mathbb{D}\rightarrow M$,
\begin{equation}
\mu_g=gd\zeta\otimes d\overline{\zeta}\quad \mbox{with}\quad g:=\varphi^\ast F^2\label{iph}
\end{equation}
 is a pseudohermitian metric on $\mathbb{D}$.
\end{proposition}
\begin{proof}
It suffices to show that $g:=\varphi^\ast F^2$ is a non-negative upper semicontinuous function on $\mathbb{D}$ with $S_g:=g^{-1}(0)$ being discrete.
Denote $\varphi=(\varphi_1,\cdots,\varphi_n)$. Then
$$\varphi_\ast\left(\lambda\frac{\partial}{\partial \zeta}\right)=\lambda\sum_{j=1}^n\varphi_j'(\zeta)\frac{\partial}{\partial z^j}\in T_{\varphi(\zeta)}^{1,0}M,\quad\forall\lambda\in\mathbb{C}.$$
 By assumption, $F$ is an upper semicontinuous complex Finsler metric on $M$,  $\varphi$ and $\varphi'$ are holomorphic maps, it follows that
$$g(\zeta):=(\varphi^\ast F^2)(\zeta)=F^2(\varphi(\zeta);\varphi'(\zeta))$$
 is  a non-negative upper semicontinuous function on $\mathbb{D}$.
 Since $g(\zeta)=0$ if and only if $\varphi_\ast(\frac{\partial}{\partial \zeta})=0$, which is equivalent to $\varphi_1'(\zeta)=\cdots=\varphi_n'(\zeta)=0$. The set
$$S_g=g^{-1}(0)={g_1}^{-1}(0)\cap \cdots\cap{g_n}^{-1}(0)$$  is a discrete subset of $\mathbb{D}$, where $g_j:=\varphi_j'$ for $j=1,\cdots,n$. This completes the proof.
\end{proof}
\begin{remark}
If $F:T^{1,0}M\rightarrow[0,+\infty)$  is  a strongly pseudoconvex complex Finsler metric and  $\varphi:\mathbb{D}\rightarrow M$ is a holomorphic embedding,  then $\mu_g$ defined by \eqref{iph} is  clearly a Hermitian metric on $\mathbb{D}$, namely $\varphi^\ast F^2$ is a positive smooth function on $\mathbb{D}$.
\end{remark}

\begin{definition}[\cite{AP}]
The (lower) generalized Laplacian of an upper semicontinuous function $u$ (defined in a neighborhood of  $\zeta\in\mathbb{D}$) is defined by
\begin{equation}
\triangle u(\zeta)=4\liminf_{r\rightarrow 0}\,\frac{1}{r^2}\left\{\frac{1}{2\pi}\int_0^{2\pi}[u(\zeta+re^{i\theta})-u(\zeta)]d\theta\right\}.\label{GL}
\end{equation}
\end{definition}
\begin{remark}\label{maxpoint}
If $u$ is $C^2$ in a neighborhood of the point $\zeta$, then \eqref{GL} actually reduces to $\triangle u(\zeta)=4\frac{\partial^2u}{\partial\zeta\partial\overline{\zeta}}$. Second, $\triangle u\geq 0$ if and only if $u$ is subharmonic. Finally, if $\zeta_0$ is a point of local maximum for $u$, then $\triangle u(\zeta_0)\leq 0$.
\end{remark}
\begin{definition}[\cite{AP}]\label{def-GC}
Let $\mu_g=g(\zeta)d\zeta\otimes d\overline{\zeta}$ be a pseudohermitian metric on $\mathbb{D}$. Then the Gaussian curvature $K(\mu_g)$ of $\mu_g$  is the function
\begin{equation}
K(\mu_g)(\zeta)=-\frac{1}{2g(\zeta)}\triangle \log g(\zeta) \label{GC}
\end{equation}
defined on $\mathbb{D}\setminus S_g$, where $\triangle$ is the (lower) generalized Laplacian given by \eqref{GL}. Clearly, if $\mu_g$ is a standard Hermitian metric, $K(\mu_g)$ reduces to the usual Gaussian curvature with $\triangle=4\frac{\partial^2}{\partial\zeta\partial\overline{\zeta}}$.
\end{definition}

\begin{definition}[\cite{AP}]\label{def-HC}
Let $F:T^{1,0}M\rightarrow[0,+\infty)$ be an upper semicontinuous complex Finsler metric on a complex manifold $M$. Take $p\in M$ and $v\in T_p^{1,0}M$ with $v\neq 0$.
The holomorphic curvature $K_F$ of $F$ at $p$ in the direction $v$ is given by
\begin{equation}
K_F(p;v)=\sup_{\varphi}\left\{K(\varphi^\ast F^2)(0)\right\},  \label{HCD}
\end{equation}
where the supremum is taken with respect to the family of all holomorphic maps $\varphi:\mathbb{D}\rightarrow M$ with
$\varphi(0)=p$ and $\varphi'(0)=\lambda v$ for some $\lambda\in\mathbb{C}^\ast:=\mathbb{C}\setminus\{0\}$,  $K(\varphi^\ast F^2)(0)$ is the Gaussian curvature of the pseudohermitian metric $\varphi^\ast F^2$ at the point $0\in \mathbb{D}$.
\end{definition}

\begin{remark}
$K_F(p;v)$ depends only on the complex line in $T_p^{1,0}M$ spanned by $v$, namely $K_F(p;v)=K_F(p;[v])$ for any $(p;[v])\in PT^{1,0}M$.
If $F$ is a strongly pseudoconvex complex Finsler metric on $M$, then $K_F$ is a smooth function on $PT^{1,0}M$; if furthermore $M$ is a homogeneous complex manifold
and $F$ is an $\mbox{Aut}(M)$-invairant strongly pseudoconvex complex Finsler metric on $M$, then $F$ has bounded holomorphic sectional curvature,
namely $K_F$ is bounded below and above by two constants, respectively.
\end{remark}
Note that any holomorphic map $\varphi$ satisfying Definition \ref{def-HC} is a  local embedding in a neighborhood of the origin, and by \eqref{HCD},
\begin{equation}
K(\varphi^\ast F^2)(0)\leq K_F(p;v).\label{IGC}
\end{equation}

\begin{proposition}\label{prop-hg}
Suppose $f:\mathbb{D}\to[0,+\infty)$ is an upper semicontinuous function and $g:\mathbb{D}\to(0,+\infty)$ is continuous. Define
\[
h(z):=\log\frac{f(z)}{g(z)},\qquad z\in \mathbb{D},
\]
with the convention $\log 0=-\infty$. Then $h$ is an upper semicontinuous function on $\mathbb{D}$.
\end{proposition}

\begin{proof}
Since $g>0$ and continuous, $1/g$ is continuous. The product $\psi(z):=f(z)/g(z)=f(z)\cdot (1/g(z))$ is upper semicontinuous.
Fix $z_0\in \mathbb{D}$. If $\psi(z_0)>0$, then by upper semicontinuity of $\psi$,
\[
\limsup_{z\to z_0} \psi(z)\le \psi(z_0).
\]
Since $\log x$ is continuous and strictly increasing on $(0,\infty)$, we get
\[
\limsup_{z\to z_0} \log \psi(z)
= \log\left(\limsup_{z\to z_0} \psi(z)\right)
\le \log \psi(z_0).
\]
If $\psi(z_0)=0$, upper semicontinuity gives $\displaystyle\limsup_{z\to z_0}\psi(z)\le 0$. As $\psi\ge 0$, this implies $\psi(z)\to 0$ as $z\to z_0$, hence $\log \psi(z)\to -\infty$.
Thus
\[
\limsup_{z\to z_0} \log \psi(z) = -\infty = \log \psi(z_0).
\]
In both cases, $\displaystyle\limsup_{z\to z_0} h(z)\le h(z_0)$. Therefore $h$ is upper semicontinuous.
\end{proof}

\begin{remark}
The positivity assumption on $g$ is essential. Indeed, take $f(z)\equiv 1$ and $g(z)=\vert z\vert$. Then $f$ is continuous (hence upper semicontinuous), $g$ is continuous and nonnegative, but
\[
\log\frac{f(z)}{g(z)}=-\log \vert z\vert,\qquad z\neq 0,
\]
which tends to $+\infty$ as $z\to 0$. If we define any finite value at $z=0$, the upper limit at $0$ is $+\infty$, exceeding that value; hence the function is not upper semicontinuous on $D$.
\end{remark}

The following proposition was proved in Zhong \cite{Zhong-c}, where it is tacitly assumed that the holomorphic curvatures of \(F_1\) and \(F_2\) are bounded. Since we will continue to use this proposition in the present paper, we add this assumption and give a detailed proof to correct it as follows.

\begin{proposition}[\cite{Zhong-c}]\label{prop-1}
Let $M$ be a complex manifold of complex dimension $n$. Suppose that $F_1$ is an upper semicontinuous complex Finsler metric and $F_2$  a continuous complex Finsler metric
 on $M$, both of them have bounded holomorphic curvatures. Then for any $(p;[v])\in PT^{1,0}M$, one has
\begin{equation}
\left(K_{F_2}-\frac{F_1^2}{F_2^2}K_{F_1}\right)(p;[v])\leq \sup_{\varphi\in\mathcal{F}_v}\left\{\frac{1}{2(\varphi^\ast F_2^2)(0)}\triangle\left(\log \frac{\varphi^\ast F_1^2}{\varphi^\ast F_2^2}\right)(0)\right\},\label{cin}
\end{equation}
where the supremum is taken with respect to the family of all holomorphic mappings
$$\mathcal{F}_v:=\{\varphi:\mathbb{D}\rightarrow M\; \mbox{is holomorphic}:\varphi(0)=p,\varphi'(0)=\lambda v\;\mbox{for some}\;\lambda\in\mathbb{C}^\ast\}. $$
\end{proposition}

\begin{proof}
By assumption and Proposition \ref{prop-iph},  $(\varphi^\ast F_i^2)(\zeta)d\zeta\otimes d\overline{\zeta}$ are pseudohermitian metrics on $\mathbb{D}$ for $i=1,2$.
Since $\varphi'(0)\neq 0$, it follows that in a neighborhood of the origin $0\in\mathbb{D}$, $\varphi^\ast F_1^2$ is positive and upper semicontinuous while
 $\varphi^\ast F_2^2$ is positive and continuous. Hence $\log \varphi^\ast F_1^2$ is upper semicontinuous and $\log \varphi^\ast F_2^2$ is continuous in a neighborhood of the origin.
So that $\triangle \log \frac{\varphi^\ast F_1^2}{\varphi^\ast F_2^2}$ makes sense for the (lower) generalized Laplacian $\triangle$ in a neighborhood of the origin $0\in\mathbb{D}$.
Thus by Definition \ref{def-GC} and \ref{def-HC}, and Proposition \ref{prop-hg}, it follows that
 at the point $(p;[v])$ we have
\begin{eqnarray*}
\left(K_{F_2}-\frac{F_1^2}{F_2^2}K_{F_1}\right)
&=&\sup_{\varphi\in\mathcal{F}_v}\left\{K(\varphi^\ast F_2^2)(0)\right\}-\frac{F_1^2}{F_2^2}(p;[v])\sup_{\varphi\in\mathcal{F}_v}\left\{K(\varphi^\ast F_1^2)(0)\right\}.
\end{eqnarray*}
Since for each $\varphi\in\mathcal{F}_v$,
$$\frac{(\varphi^\ast F_1^2)(0)}{(\varphi^\ast F_2^2)(0)}=\frac{F_1^2}{F_2^2}(p;[v]):=c,$$
is a positive number which is independent of  $\varphi$,
and by assumption, $K_{F_1}$ and $K_{F_2}$ are bounded, thus by the inequality $\sup A-c\sup B\leq \sup(A-cB)$,  at $(p;[v])$ we have
\begin{eqnarray*}
\left(K_{F_2}-\frac{F_1^2}{F_2^2}K_{F_1}\right)
&\leq&\sup_{\varphi\in\mathcal{F}_v}\left\{K(\varphi^\ast F_2^2)(0)-\frac{(\varphi^\ast F_1^2)(0)}{(\varphi^\ast F_2^2)(0)}K(\varphi^\ast F_1^2)(0)\right\}\\
&=&\sup_{\varphi\in\mathcal{F}_v}\left\{\frac{1}{2(\varphi^\ast F_2^2)(0)}\triangle \left(\log\frac{\varphi^\ast F_1^2}{\varphi^\ast F_2^2}\right)(0)\right\},
\end{eqnarray*}
\end{proof}

\begin{corollary}\label{cor-max}
Let $M$ be a complex manifold of complex dimension $n$,  $F_1$  and $F_2$ are  continuous complex Finsler metrics on $M$. Suppose that there exists a point $(p_0;[v_0])\in PT^{1,0}M$ such that  $\frac{F_1^2}{F_2^2}$ attains its  maximum at the point $(p_0;[v_0])$.
 \begin{enumerate}
 \item[(i)] If the holomorphic curvature of $F_1$ is bounded above by a negative constant $-K_1<0$ and the holomorphic curvature of  $F_2$ is bounded below by a negative constant $-K_2<0$, then
$$
\frac{F_1^2(p;v)}{F_2^2(p;v)}\leq \frac{K_2}{K_1},\quad \forall (p;[v])\in PT^{1,0}M.
$$
\item[(ii)] If the holomorphic curvature of $F_1$ is bounded above by a positive constant $K_1$ and the holomorphic curvature of  $F_2$ is bounded below by a positive constant $K_2$, then
    $$
    \frac{F_1^2(p_0;v_0)}{F_2^2(p_0;v_0)}\geq \frac{K_2}{K_1}.
    $$
\end{enumerate}

\end{corollary}

\begin{proof}
By assumption, $\varphi^\ast \left(\frac{F_1^2}{F_2^2}\right)(\zeta)$ achieves its  maximum at the point $\zeta=0$, where $\varphi:\mathbb{D}\rightarrow M$ is any holomorphic map satisfying $\varphi(0)=p_0$ and $\varphi'(0)=\lambda v_0$ for some $\lambda\in\mathbb{C}^\ast$.
Thus by Remark \ref{maxpoint}, $\triangle\left(\log \frac{\varphi^\ast F_1^2}{\varphi^\ast F_2^2}\right)(0)\leq 0$, so that
$$
\sup_{\varphi\in\mathcal{F}_v}\left\{\frac{1}{2(\varphi^\ast F_2^2)(0)}\triangle\left(\log \frac{\varphi^\ast F_1^2}{\varphi^\ast F_2^2}\right)(0)\right\}\leq 0,
$$
since $(\varphi^\ast F_2^2)(0)>0$. By Proposition \ref{prop-1}, we have
\begin{equation}
K_{F_2}(p_0;[v_0])\leq \frac{F_1^2(p_0;v_0)}{F_2^2(p_0;v_0)}K_{F_1}(p_0;[v_0]).\label{p0v0}
\end{equation}
By assumption (i) and \eqref{p0v0}, we have
$$
-K_2\leq K_{F_2}(p_0;[v_0])\leq \frac{F_1^2(p_0;v_0)}{F_2^2(p_0;v_0)}K_{F_1}(p_0;[v_0])\leq \frac{F_1^2(p_0;v_0)}{F_2^2(p_0;v_0)}(-K_1),
$$
from which it follows that
$$\frac{F_1^2(p;v)}{F_2^2(p;v)}\leq \frac{F_1^2(p_0;v_0)}{F_2^2(p_0;v_0)}\leq \frac{K_2}{K_1},\quad \forall (p;[v])\in PT^{1,0}M$$
since $\frac{F_1^2}{F_2^2}$ achieves its maximum at the point $(p_0;v_0)$.
By assumption (ii) and \eqref{p0v0}, we immediately have
$$
\frac{F_1^2(p_0;v_0)}{F_2^2(p_0;v_0)}\geq \frac{K_2}{K_1}.
$$
\end{proof}

\begin{lemma}\label{lem-DM}
Let $\mathbb{D}$ be the open unit disk in $\mathbb{C}$ which is endowed with an $\mbox{Aut}(\mathbb{D})$-invariant K\"ahler metric $\mathcal{P}(z;v)=\frac{2}{\sqrt{K_1}}\frac{\vert v\vert}{1-\vert z\vert^2}$ with constant Gaussian curvature $-K_1<0$.
Let $M$ be a complex manifold which is endowed with a strongly pseudoconvex complex Finsler metric $F$ such that the holomorphic sectional curvature $K_F$ of $F$ is bounded above by a constant $-K_2<0$, then every holomorphic map $f:\mathbb{D}\rightarrow M$ satisfies
\begin{equation}
(f^\ast F)(z;v)\leq \sqrt{\frac{K_1}{K_2}}\mathcal{P}(z;v),\quad \forall (z;v)\in T^{1,0}\mathbb{D}\cong \mathbb{D}\times\mathbb{C}.\label{DM}
\end{equation}
\end{lemma}
\begin{proof}
Let $P(z;v)=\frac{2\vert v\vert}{1-\vert z\vert^2}$ be the standard Poincar\'e metric with constant Gaussian curvature $-1$. By assumption, $\sqrt{K_2}F$ is a strongly pseudoconvex complex Finsler metric with holomorphic sectional curvature bounded above by $-1$. Thus by
Theorem 6.3 in \cite{K}, we have
\begin{eqnarray*}
f^\ast(\sqrt{K_2}F)(z;v)\leq \frac{2\vert v\vert}{1-\vert z\vert^2}=\sqrt{K_1}\cdot \frac{2}{\sqrt{K_1}}\cdot \frac{\vert v\vert}{1-\vert z\vert^2},
\end{eqnarray*}
from which we get
\begin{eqnarray*}
(f^\ast F)(z;v)\leq\sqrt{\frac{K_1}{K_2}}\mathcal{P}(z;v).
\end{eqnarray*}
This is exactly the inequality \eqref{DM}.
\end{proof}

Let $\mathfrak{D}$ be an arbitrary bounded symmetric domain. The following lemma is the key bridge between an $\mbox{Aut}(\mathfrak{D})$-invariant strongly pseudoconvex complex Finsler metric $F$ on $\mathfrak{D}$ and the Carath\'eodory metric $F_C$ (hence the Kobayashi metric $F_K$ since they coincide) on $\mathfrak{D}$.

\begin{theorem}\label{thm:comparison}
Let $\mathfrak{D}$ be a bounded symmetric domain in its Harish-Chandra realization and let $F:T^{1,0}\mathfrak{D}\to[0,+\infty)$ be an $\operatorname{Aut}(\mathfrak{D})$-invariant strongly pseudoconvex complex Finsler metric. Suppose
\begin{equation}
\min_{(z;[v])\in PT^{1,0}\mathfrak{D}}K_F(z;[v])=-K_1<0,\qquad
\max_{(z;[v])\in PT^{1,0}\mathfrak{D}}K_F(z;[v])=-K_2<0. \label{thm-c1}
\end{equation}
Then for every $(z;v)\in T^{1,0}\mathfrak{D}$,
\begin{equation}
\frac{4}{K_1}F_C^2(z;v) \le F^2(z;v) \le \frac{4}{K_2}F_C^2(z;v). \label{thm-c2}
\end{equation}
Moreover, the constants in \eqref{thm-c2} are sharp for each given $F$, namely
\begin{equation}
\frac{4}{K_1} = \min_{[v]\in PT_0^{1,0}\mathfrak{D}}\frac{F^2}{F_C^2}(0;[v]),\qquad
\frac{4}{K_2} = \max_{[v]\in PT_0^{1,0}\mathfrak{D}}\frac{F^2}{F_C^2}(0;[v]). \label{thm-c3}
\end{equation}
\end{theorem}

\begin{proof}
We prove the upper and lower bounds separately.

Since $\mathfrak{D}$ is a homogeneous domain, for any fixed $z_0\in \mathfrak{D}$ and any $v_0\in T_{z_0}^{1,0}\mathfrak{D}$, there exists a $\varphi_{z_0}\in\operatorname{Aut}(\mathfrak{D})$ such that $\varphi_{z_0}(z_0)=0$. Let $w:=d\varphi_{z_0}(v)\in T_0^{1,0}\mathfrak{D}$. By the $\operatorname{Aut}(\mathfrak{D})$-invariance of both $F$ and $F_C$,
\[
F(z_0;v)=F(0;w),\qquad F_C(z_0;v)=F_C(0;w).
\]
Thus it suffices to prove \eqref{thm-c2} for all nonzero vectors $w\in T_0^{1,0}\mathfrak{D}$.

Consider the ratio $\phi:=F_C^2/F^2$. Then $\phi$ attains its maximum at some point $(0;[v_1])\in PT_0^{1,0}\mathfrak{D}$ since $\mathfrak{D}$ is homogeneous, $\phi$ is well-defined and  continuous on $PT_0^{1,0}\mathfrak{D}$. By assertion (i) in Corollary \ref{cor-max}, we have
$$
\frac{F_C^2}{F^2}(0;[v])\leq\frac{-K_1}{-4}=\frac{K_1}{4},\quad\forall (0;[v])\in PT_0^{1,0}\mathfrak{D},
$$
from which it follow that $\frac{4}{K_1}F_C^2(0;v)\leq F^2(0;v)$ for any nonzero vectors $v\in T_0^{1,0}\mathfrak{D}$.
To prove the  second inequality $F^2(0;v)\leq \frac{4}{K_2}F_C^2(0;v)$ in \eqref{thm-c2}. It suffices to note that $1/\phi$ also attains its maximum at some point $(0;[v_2])\in PT_0^{1,0}\mathfrak{D}$ since $\mathfrak{D}$ is homogeneous, $1/\phi$ is well-defined and continuous on $PT_0^{1,0}\mathfrak{D}$. Thus by assertions (i) in Corollary \ref{cor-max}, we have
$$
\frac{F^2}{F_C^2}(0;[v])\leq \frac{-4}{-K_2}=\frac{4}{K_2},\quad \forall(0;[v])\in PT_0^{1,0}\mathfrak{D},
$$
from which it follows that $F^2(0;v)\leq \frac{4}{K_2}F_C^2(0;v)$ for any nonzero vectors $v\in T_0^{1,0}\mathfrak{D}$.

The equalities in \eqref{thm-c3} follows immediately since $\mathfrak{D}$ is homogeneous, both $\phi$ and $1/\phi$ are well-defined and positive continuous functions on $PT_0^{1,0}\mathfrak{D}$.

This completes the proof.
\end{proof}

\begin{proposition}\label{prop:rigidity}
Let $D\subset\mathbb C^n$ be a bounded convex circular domain with $0\in D$, and let $\pmb{f}:\mathbb{C}^n\rightarrow[0,+\infty)$ be a complex norm on $\mathbb C^n$ whose unit ball is $D$. Suppose $g:D\to D$ is a holomorphic map with $g(0)=0$ such that
\begin{equation}
\pmb{f}(g_*(v))=\pmb{f}(v),\qquad \forall v\in\mathbb{C}^n.\label{cmb}
\end{equation}
Then $g\in\operatorname{Aut}(D)$.
\end{proposition}

\begin{proof}
Set \(L:=g_*(0)\).  By assumption, \(D\) is the unit ball of the complex norm \(\pmb f\), we have
\[
D=\{v\in\mathbb C^n:\pmb f(v)<1\}.
\]
From \eqref{cmb},
\[
\pmb f(Lv)=\pmb f(v),\qquad \forall v\in\mathbb C^n.
\]
Hence \(L\) preserves the norm \(\pmb f\). Therefore
\[
v\in D
\iff \pmb f(v)<1
\iff \pmb f(Lv)<1
\iff Lv\in D.
\]
Thus \(L(D)\subset D\).

Moreover, \(\pmb f(Lv)=\pmb f(v)\) implies that \(L\) is injective. Indeed, if \(Lv=0\), then
\[
\pmb f(v)=\pmb f(Lv)=\pmb f(0)=0,
\]
and since \(\pmb f\) is a norm, this gives \(v=0\). As \(L:\mathbb C^n\to\mathbb C^n\) is an injective linear map between finite-dimensional spaces, it is invertible; that is, \(L\in \mathrm{GL}(n;\mathbb C)\).

Now take any \(w\in D\). Put \(v:=L^{-1}w\). Then
\[
\pmb f(v)=\pmb f(Lv)=\pmb f(w)<1,
\]
so \(v\in D\), and \(w=Lv\in L(D)\). Hence \(D\subset L(D)\). Combining this with \(L(D)\subset D\), we obtain
\[
L(D)=D.
\]
Consequently, \(L\) is a linear automorphism of \(D\), i.e. \(L\in\operatorname{Aut}(D)\).

Define
\[
h:=L^{-1}\circ g:D\to D.
\]
Then \(h\) is holomorphic, \(h(0)=0\), and
\[
h_*(0)=L^{-1}\circ g_*(0)=\mathrm{Id}_{\mathbb C^n}.
\]
Since \(D\) is a bounded domain, Cartan's uniqueness theorem implies that
\[
h=\mathrm{Id}_D.
\]
Therefore \(g=L\) is complex linear. As \(L(D)=D\), we have \(g\in\operatorname{Aut}(D)\). This completes the proof.
\end{proof}

\section{Proof of Theorem \ref{thm-a}}

Let \(\tilde{M}\) be the complement of the zero section in \(TM\). Let \(\pi:\tilde{M}\to M\) be the projection, and let \(\pi^*TM\) be the pullback real tangent bundle over \(\tilde{M}\). Since \(F\) is strongly convex, it induces both a Riemannian metric \(\langle\cdot\,\vert\,\cdot\rangle\) along each fiber of \(\pi^*TM\). The Chern connection of \(F\) is a real linear connection
\[
\overset{\mathrm{Ch}}{\nabla}:\Gamma(\pi^*TM)\to\Gamma(T^*\tilde{M}\otimes\pi^*TM)
\]
that is torsion-free and real horizontally metrical with respect to $\langle\cdot\,\vert\,\cdot\rangle$(cf. Bao--Chern--Shen \cite{BCS}).

\begin{proof}[Proof of Theorem \ref{thm-a}]
By Theorem 5.8 in \cite{Deng}, \(F\) is necessarily a real Berwald metric on \(M\); hence its real Berwald connection \(\overset{\mathrm{B}}{\nabla}\) coincides with its Chern connection \(\overset{\mathrm{Ch}}{\nabla}\), and both coincide with the pullback of the Levi-Civita connection \(\overset{\mathrm{LC}}{\nabla}\) of some \(G\)-invariant Riemannian metric \(g\) on \(M\). Moreover, \((M,g)\) is a globally symmetric Riemannian space.

On the other hand, by Theorem 5.24 in \cite{Deng}, there also exists a \(G\)-invariant Hermitian metric \(Q\) on \(M\) such that \((M,J,Q)\) is a Hermitian symmetric space. Hence, by Proposition \ref{Helgason}, \(Q\) is necessarily a K\"ahler metric on \(M\).

Since all \(G\)-invariant Riemannian metrics on the symmetric homogeneous space \(G/H\) have the same Levi-Civita connection by Theorem 5.10 in \cite{Deng}, we have
\begin{equation}
\overset{\mathrm{B}}{\nabla}=\overset{\mathrm{Ch}}{\nabla}=\overset{\mathrm{LC}}{\nabla}=\overset{Q}{\nabla},\qquad \overset{Q}{\nabla}J=0,\label{connection}
\end{equation}
where \(\overset{Q}{\nabla}\) is the Levi-Civita connection of \(Q\). Hence, by \eqref{connection}, we have \(\overset{\mathrm{Ch}}{\nabla}\mathcal{J}=0\), where \(\mathcal{J}=J\circ\pi\) denotes the pullback of \(J\) to \(\pi^*TM\).

Since \(\overset{\mathrm{Ch}}{\nabla}\) coincides with the horizontal part of the Cartan connection \(\overset{\mathrm{Car}}{\nabla}\), and \(\overset{\mathrm{Ch}}{\nabla}\mathcal{J}=0\) is equivalent to \(\overset{\mathrm{Car}}{\nabla}_X\boldsymbol{J}=0\) for any \(X\in\mathcal{H}_{\mathbb{R}}\), where \(\boldsymbol{J}\) is the complex structure on \(T^{1,0}M\) induced by \(J\) \cite{XZ}. Thus, by Theorem 3.11 in \cite{XZ}, \(F\) is necessarily a K\"ahler--Berwald metric.
\end{proof}

\section{Proof of Theorem \ref{thm-b}}

In this section, we shall complete the proof of Theorem \ref{thm-b} by proving the following Theorem \ref{thm:schwarz} and Theorem \ref{thm:optimal-Lu}, respectively.

\begin{theorem}\label{thm:schwarz}
Let $\mathfrak{D}\subset\mathbb{C}^n$ be a bounded symmetric domain in its Harish-Chandra realization and let $F:T^{1,0}\mathfrak{D}\rightarrow[0,+\infty)$ be an $\mbox{Aut}(\mathfrak{D})$-invariant K\"ahler-Berwald metric such that its holomorphic sectional curvature is bounded below and above by negative constants $-K_1<0$ and $-K_2<0$, respectively. Then for every holomorphic mapping $f:\mathfrak{D}\to \mathfrak{D}$, we have
\begin{equation}
(f^*F)(z;v) \le \sqrt{\frac{K_1}{K_2}}\, F(z;v),\quad \forall(z;v)\in T^{1,0} \mathfrak{D}.
\label{thm-s1}
\end{equation}
Moreover, if there exist a point $z_0\in \mathfrak{D}$ and for all nonzero vectors $v\in T_{z_0}^{1,0}\mathfrak{D}$ such that
\begin{equation}
(f^*F)(z_0;v)=F(z_0;v), \label{thm-eq}
\end{equation}
then $f\in\operatorname{Aut}(\mathfrak{D})$.
\end{theorem}

\begin{proof}
It suffices to show that for any fixed point $z_0\in \mathfrak{D}$, the inequality \eqref{thm-s1} holds for all nonzero vectors $v\in T_{z_0}^{1,0}\mathfrak{D}$ and any holomorphic map $f:\mathfrak{D}\rightarrow \mathfrak{D}$.

First,  since $\mbox{Aut}(\mathfrak{D})$ acts transitively on $\mathfrak{D}$, one can choose $\varphi,\psi\in\operatorname{Aut}(\mathfrak{D})$ such that $\varphi(z_0)=0$ and $\psi(f(z_0))=0$. Let $\tilde{v}:=d\varphi(v)\in T_0^{1,0}\mathfrak{D}$ and define
\[
g := \psi \circ f \circ \varphi^{-1}.
\]
Then $g:\mathfrak{D}\to \mathfrak{D}$ is holomorphic and, by the $\operatorname{Aut}(\mathfrak{D})$-invariance of $F$, we have
\[
(f^*F)(z_0;v) = (g^*F)(0;\tilde{v}),\qquad F(z_0;v)=F(0;\tilde{v}).
\]
Thus it suffices to prove \eqref{thm-s1} at $z_0=0$ for holomorphic maps $g:\mathfrak{D}\to \mathfrak{D}$ satisfying $g(0)=0$ and for any nonzero vectors $\tilde{v}\in T_0^{1,0}\mathfrak{D}$.

Let's fix $0\neq \tilde{v}\in T_0^{1,0}\mathfrak{D}$. By Theorem \ref{thm:comparison}, we have
\begin{equation}
F(0;\tilde{v}) \ge \frac{2}{\sqrt{K_1}}F_C(0;\tilde{v}). \label{thm-s3}
\end{equation}
Define
\begin{equation}
\widetilde{w} := \frac{2}{\sqrt{K_1}F(0;\tilde{v})}\tilde{v} \in T_0^{1,0}\mathfrak{D}. \label{thm-s4}
\end{equation}
Then $F(0;\widetilde{w})=\frac{2}{\sqrt{K_1}}$. Thus by \eqref{thm-s3}, $F_C(0;\widetilde{w})\le 1$.

By Theorem B in Suzuki \cite{Suzuki} for bounded convex complete circular domains in $\mathbb{C}^n$ with center $0$, the unit ball of the Carath\'eodory metric $F_C$ at the tangent space $T_0^{1,0}\mathfrak{D}$ satisfies
\[
I_0(F_C):=\left\{w\in T_0^{1,0}\mathfrak{D} : F_C(0;w)<1\right\} =\mathfrak{D}.
\]
 Hence $F_C(0;\widetilde{w})\le 1$ implies that $\widetilde{w}\in \overline{\mathfrak{D}}$, the closure of $\mathfrak{D}$ (identifying $T_0^{1,0}\mathfrak{D}$ with $\mathbb{C}^n$). Thus we can define the following holomorphic embedding:
\begin{equation}
j_{\widetilde{w}}:\mathbb{\mathfrak{D}}\to \mathfrak{D},\qquad j_{\widetilde{w}}(\zeta):=\zeta \widetilde{w}. \label{thm-s5}
\end{equation}
It's clear that $j_{\widetilde{w}}(0)=0$ and $j_{\widetilde{w}}'(0)=\widetilde{w}$.

Let $\mathcal P(z;v)$ be the $\mbox{Aut}(\mathbb{\mathfrak{D}})$-invariant K\"ahler metric with constant holomorphic sectional curvature $-K_1<0$ (see Lemma \ref{lem-DM}). Then
\begin{equation}
\mathcal P\left(\zeta;1\right):=\frac{2}{\sqrt{K_1}}\frac{1}{1-\vert\zeta\vert^2}. \label{thm-s6}
\end{equation}
By construction,
\begin{equation}
(j_{\widetilde{w}}^*F)\left(0;1\right)=F(0;\widetilde{w})=\frac{2}{\sqrt{K_1}}
=\mathcal P\left(0;1\right). \label{thm-s7}
\end{equation}

Now apply Lemma \ref{lem-DM}  to the holomorphic map $h:=g\circ j_{\widetilde{w}}:\mathbb{\mathfrak{D}}\to \mathfrak{D}$. We obtain
\begin{equation}
(h^*F)\left(0;1\right)
\le \sqrt{\frac{K_1}{K_2}}\,
\mathcal P\left(0;1\right). \label{thm-s8}
\end{equation}
Using \eqref{thm-s7}, this gives
\begin{equation}
(g^*F)(0;\widetilde{w})\le \sqrt{\frac{K_1}{K_2}}F(0;\widetilde{w})=\sqrt{\frac{K_1}{K_2}}\frac{2}{\sqrt{K_1}}=\frac{2}{\sqrt{K_2}}. \label{thm-s9}
\end{equation}

On the other hand, $\tilde{v}=\frac{\sqrt{K_1}F(0;\tilde{v})}{2}\widetilde{w}$, thus
by the homogeneity of $F$  and \eqref{thm-s4}, we obtain
\begin{equation}
(g^*F)(0;\tilde{v})
= \frac{\sqrt{K_1}F(0;\tilde{v})}{2}(g^*F)(0;\widetilde{w})
\le \sqrt{\frac{K_1}{K_2}}F(0;\tilde{v})\label{thm-s10}
\end{equation}
holds for $\tilde{v}$. Since $\tilde{v}\in T_0^{1,0}\mathfrak{D}$ is an arbitrary fixed nonzero vector and \eqref{thm-s10} is invariant by replacing $\tilde{v}$ with $\lambda \tilde{v}$ for any $\lambda\in\mathbb{C}^\ast$, it follows that \eqref{thm-s10} holds for all nonzero vectors $\tilde{v}\in T_0^{1,0}\mathfrak{D}$.

By the $\mbox{Aut}(\mathfrak{D})$-invariance of $F$,  we have
$$
(g^\ast F)(0;\tilde{v})=(f^\ast F)(z_0;v),\quad F(0;\tilde{v})=F(z_0;v).
$$
Thus in terms of $z_0$ and $v$, we actually have proved that the inequality \eqref{thm-s1} holds at  the fixed point $z_0\in \mathfrak{D}$ for any nonzero tangent vector $v\in T_{z_0}^{1,0}\mathfrak{D}$. Since $z_0\in \mathfrak{D}$ is an arbitrary fixed point, this completes the proof of \eqref{thm-s1}.

Now suppose that \eqref{thm-eq} holds at some  $z_0\in \mathfrak{D}$ for all nonzero tangent vectors $v\in T_{z_0}^{1,0}\mathfrak{D}$. Choose $\varphi,\psi\in\operatorname{Aut}(\mathfrak{D})$ such that $\varphi(z_0)=0$ and $\psi(f(z_0))=0$. Define
\[
g := \psi \circ f \circ \varphi^{-1}.
\]
Then $g:\mathfrak{D}\to \mathfrak{D}$ is holomorphic and $g(0)=0, \tilde{v}:=d\varphi(v)\in T_0^{1,0}\mathfrak{D}$ and $\tilde{v}\neq 0$. By the $\operatorname{Aut}(\mathfrak{D})$-invariance of $F$, \eqref{thm-eq} implies
\begin{equation}
(g^*F)(0;\tilde{v})=F(0;\tilde{v}), \quad\forall\tilde{v}\in T_0^{1,0}\mathfrak{D}.\label{thm-s11}
\end{equation}
Thus by Proposition \ref{prop:rigidity}, $g\in \mbox{Aut}(\mathfrak{D})$, this implies  $f\in\mbox{Aut}(\mathfrak{D})$, which completes the proof.
\end{proof}

\begin{remark}
Taking \(\mathfrak{D}\) to be the unit polydisk \(P_n\subset\mathbb{C}^n\) \((n\ge 2)\) and the classical domain \(\mathfrak{R}_I(m;n)\) of type I, respectively, we shall give examples in Section \ref{example} showing that if there exist a point \(z_0\in\mathfrak{D}\) and a nonzero tangent vector \(v_0\in T_{z_0}^{1,0}\mathfrak{D}\) such that
\begin{equation}
(f_0^*F)(z_0;v_0)=\sqrt{\frac{K_1}{K_2}}\,F(z_0;v_0)
\end{equation}
for some holomorphic mapping \(f_0:\mathfrak{D}\to\mathfrak{D}\), then \(f_0\notin\operatorname{Aut}(\mathfrak{D})\).
\end{remark}

\begin{definition}Let $\mathfrak{D}$ be an arbitrary bounded symmetric domain and $F$ be an $\mbox{Aut}(\mathfrak{D})$-invariant strongly pseudoconvex complex Finsler metric on $\mathfrak{D}$ such that the holomorphic sectional curvature of $F$ is bounded below and above by negative constants $-K_1<0$ and $-K_2<0$, respectively. We call
the constant
\begin{equation}
k_0(\mathfrak{D},F):=\sqrt{\frac{K_1}{K_2}} \label{eq-16}
\end{equation}
the \emph{Lu constant} of $(\mathfrak{D},F)$.
\end{definition}

\begin{theorem}\label{thm:optimal-Lu}
Let \(\mathfrak{D}\) be an arbitrary bounded symmetric domain in its Harish-Chandra realization and let
\(F:T^{1,0}\mathfrak{D}\to[0,+\infty)\) be an \(\operatorname{Aut}(\mathfrak{D})\)-invariant
strongly pseudoconvex complex Finsler metric with holomorphic sectional
curvature bounded below and above by \(-K_1<0\) and \(-K_2<0\), respectively.
The constant
\[
k_0(\mathfrak{D},F):=\sqrt{\frac{K_1}{K_2}}
\]
is optimal: there is no constant \(c<k_0(\mathfrak{D},F)\) such that
\[
(f^*F)(z;v)\le cF(z;v),\qquad \forall (z;v)\in T^{1,0}\mathfrak{D},
\]
for all holomorphic self-maps \(f:\mathfrak{D}\to \mathfrak{D}\).
\end{theorem}

\begin{proof}
By Theorem \ref{thm:comparison}, there exist directions
\(v_{\min},v_{\max}\in T_0^{1,0}\mathfrak{D}\) such that
\[
F_C(0;v_{\min})=F_C(0;v_{\max})=1,
\]
and
\[
F(0;v_{\min})=\frac{2}{\sqrt{K_1}},\qquad
F(0;v_{\max})=\frac{2}{\sqrt{K_2}}.
\]
In the Harish-Chandra realization, \(\mathfrak{D}\) is a bounded convex circular
domain containing \(0\). By Theorem~B in Suzuki \cite{Suzuki},
the indicatrix of the Carath\'eodory metric \(F_C\) at \(0\) satisfies
\[
I_0(F_C):=\left\{w\in T_0^{1,0}\mathfrak{D}: F_C(0;w)<1\right\}=\mathfrak{D}.
\]
Since \(F_C(0;\cdot)\) is a complex norm on \(T_0^{1,0}\mathfrak{D}\), by the
Hahn--Banach theorem there exists a complex linear functional
\[
\ell:T_0^{1,0}\mathfrak{D}\to\mathbb C
\]
such that
\[
\ell(v_{\min})=1,\qquad \vert \ell(\xi)\vert\le F_C(0;\xi),\quad
\forall \xi\in T_0^{1,0}\mathfrak{D}.
\]
Define a holomorphic map
\[
L:\mathfrak{D}\to \mathfrak{D},\qquad L(z):=\ell(z)\,v_{\max},
\]
where we identify \(z\in\mathfrak{D}\subset\mathbb C^n\) with a vector in
\(T_0^{1,0}\mathfrak{D}\cong\mathbb C^n\).
If \(z\in \mathfrak{D}\), then \(F_C(0;z)<1\), and hence
\[
\vert\ell(z)\vert\le F_C(0;z)<1.
\]
Therefore
\[
F_C(0;L(z))
=F_C\bigl(0;\ell(z)v_{\max}\bigr)
=\vert\ell(z)\vert F_C(0;v_{\max})
=\vert\ell(z)\vert<1.
\]
Thus \(L(z)\in I_0(F_C)=\mathfrak{D}\), so \(L\) maps \(\mathfrak{D}\) into itself.

Moreover, since \(\ell(v_{\min})=1\), we have
$
L(v_{\min})=v_{\max},
$
and hence
$
dL_0(v_{\min})=v_{\max}.
$
Consequently,
\[
(L^*F)(0;v_{\min})
=F(0;dL_0(v_{\min}))
=F(0;v_{\max})
=\frac{2}{\sqrt{K_2}}.
\]
On the other hand,
\[
F(0;v_{\min})=\frac{2}{\sqrt{K_1}}.
\]
Thus
\[
\frac{(L^*F)(0;v_{\min})}{F(0;v_{\min})}
=
\frac{2/\sqrt{K_2}}{2/\sqrt{K_1}}
=
\sqrt{\frac{K_1}{K_2}}
=
k_0(\mathfrak{D},F).
\]
Taking \(f_0:=L\) and \(v_0:=v_{\min}\), we obtain
\[
(f_0^*F)(0;v_0)=k_0(\mathfrak{D},F)F(0;v_0).
\]
Now suppose there exists a constant \(c<k_0(\mathfrak{D},F)\) such that
\[
(f^*F)(z;v)\le cF(z;v)
\]
for all holomorphic self-maps \(f:\mathfrak{D}\to\mathfrak{D}\) and all
\((z;v)\in T^{1,0}\mathfrak{D}\). Applying this to \(f_0\) and \(v_0\), we get
\[
(f_0^*F)(0;v_0)\le cF(0;v_0).
\]
But the left-hand side equals \(k_0(\mathfrak{D},F)F(0;v_0)\), which is strictly
larger than \(cF(0;v_0)\), a contradiction. Hence no smaller constant
exists, and \(k_0(\mathfrak{D},F)\) is optimal.
\end{proof}

\section{Examples for the equality in \eqref{thm-s1}}\label{example}

In this section we shall provide some explicit examples to show that equality in \eqref{thm-s1}.

\begin{example}\label{ex}\cite{Zhong-b} Let $P_n$ be the unit polydisk in $\mathbb{C}^n(n\geq 2)$ endowed with the $\mbox{Aut}(P_n)$-invariant strongly pseudoconvex complex Finsler metric
$$
F_{t,k}^2(z;v)=\frac{1}{1+t}\left\{\sum_{i=1}^n\frac{\vert v^i\vert^2}{(1-\vert z^i\vert^2)^2}+t\sqrt[k]{\sum_{i=1}^n\frac{\vert v^i\vert^{2k}}{(1-\vert z^i\vert^2)^{2k}}}\right\},
$$
which is a strongly convex K\"ahler-Berwald metric on $P_n$ for any fixed $t\in[0,+\infty)$ and integer $k\geq 2$.
It is easy to check that
$$
\max_{1\leq i\leq n}\{\vert v^i\vert^2\}\leq F_{t,k}^2(0;v)\leq \frac{n+t\sqrt[k]{n}}{1+t}\max_{1\leq i\leq n}\{\vert v^i\vert^2\},\quad \forall v\in T_0^{1,0}P_n.
$$
In \cite{Zhong-b}, Zhong showed that the minimum and maximum holomorphic sectional curvatures of $F_{t,k}$ are given by $-4$ and
$-\frac{4(1+t)}{n+t\sqrt[k]{n}}$, respectively. Thus
\begin{equation}
\frac{K_1}{K_2}=\frac{n+t\sqrt[k]{n}}{1+t}.\label{sct}
\end{equation}
Note that for $t=0$,  the Lu constant $k_0(P_n, F_{t,k})$  reduces to  $\sqrt{n}$, which was obtained by  Lu \cite{Lu-a} with respect to the Bergman metric on $P_n$. For  $t\in(0,\infty)$ and integer $k\geq 2$, however, the Lu constant $k_0(P_n, F_{t,k})$ depends both on the rank $n$ of $P_n$ and the minimum and maximum holomorphic sectional curvatures of $F_{t,k}$, hence the Lu constant $k_0(P_n, F_{t,k})$ is both an analytic invariant and geometric invariant which can be seen more clearly in complex Finsler settings.

Taking $f_0(z)=(z^1,\cdots,z^1)$, $z_0=0$ the origin of $P_n$ and $v_0:=(1,\underbrace{0,\cdots,0}_{n-1})\cong \frac{\partial}{\partial z^1}\in T_0^{1,0}P_n$, then $(f_0)_\ast(v_0)=(1,\cdots,1)\cong \sum_{i=1}^n\frac{\partial}{\partial z^i}\in T_0^{1,0}P_n$, where $(f_0)_\ast$ denotes the tangential map of $f_0$ at the origin $0\in P_n$. It is easy to check that
$$
(f_0^\ast F_{t,k})(z_0;v_0)=\sqrt{\frac{n+t\sqrt[k]{n}}{1+t}},\quad F_{t,k}(z_0;v_0)=1.
$$
So that
$$(f_0^\ast F_{t,k})(z_0;v_0)=\sqrt{\frac{n+t\sqrt[k]{n}}{1+t}}F_{t,k}(z_0;v_0).$$
But $f_0\notin\mbox{Aut}(P_n)$.
\end{example}

\begin{example}[\cite{GZ}]
Consider the classical domain $\mathfrak{R}_I(m;n)$ of type $I$, which is equipped with the following $\mbox{Aut}(\mathfrak{R}_I)$-invariant K\"ahler-Berwald metric
 \begin{eqnarray*}
			F_I^2(Z;V)&=&\frac{m+n}{1+t}\Bigg\{\mbox{tr}\left\{(I-ZZ^\ast)^{-1}V(I-Z^\ast Z)^{-1}V^\ast\right\}\\
			&&+t\sqrt[k]{\mbox{tr}\left\{\left[(I-ZZ^\ast)^{-1}V(I-Z^\ast Z)^{-1}V^\ast\right]^k\right\}}\Bigg\}.
	\end{eqnarray*}
 For this metric, by Theorem 3.26 in \cite{GZ} and Theorem 3 in \cite{Zhong-c}, we have
		\begin{eqnarray*}
			(f^\ast F_I^2)(Z;V)\leq\frac{m+t\sqrt[k]{m}}{1+t}F_I^2(Z;V),\quad \forall(Z;V)\in T^{1,0}\mathfrak{R}_I(m;n).
		\end{eqnarray*}
\begin{enumerate}
\item[(1)] If $m=1$, then $\mathfrak{R}_I=B_n$ is the open unit ball in $\mathbb{C}^n$ and $F_I$ reduces to the Bergman metric $F_B$ on $B_n$,
$$k_0(\mathfrak{R}_I, F_B)=1;$$
\item[(2)] If $m\geq 2$ and $t=0$, then $F_I$ reduces to the Bergman metric $F_B$ on $\mathfrak{R}_I$,
$$k_0(\mathfrak{R}_I, F_B)=\sqrt{m};$$
\item[(3)] If $m\geq 2$ and $t\in (0,+\infty)$, $F_I$ is a non-Hermitian quadratic $\mbox{Aut}(\mathfrak{R}_I)$-invariant metric on $\mathfrak{R}_I$,
$$1<k_0(\mathfrak{R}_I, F_I)=\sqrt{\frac{m+t\sqrt[k]{m}}{1+t}}<\sqrt{m}.$$
\end{enumerate}

Now define a holomorphic self-map $f_0:\mathfrak{R}_I(m;n)\to \mathfrak{R}_I(m;n)$ by
$$f_0(Z)=\begin{pmatrix}
         Z_{12} & 0 & \cdots & 0 & 0 & \cdots & 0 \\
         0 & Z_{12} & \cdots & 0 & 0 & \cdots & 0 \\
         \vdots & \vdots & \cdots & 0& \vdots & \cdots & 0 \\
         0 & 0 & \cdots & Z_{12} & 0 & \cdots & 0 \\
       \end{pmatrix}_{m\times n},
$$
i.e., $f_0(Z)$ is a diagonal matrix with all diagonal entries equal to $Z_{12}$. It is easy to check that $f_0(0)=0$ and $f_0$ maps $\mathfrak{R}_I(m;n)$ into itself provided $\vert Z_{12}\vert<1$. Take $V_0=E_{12}\in T_0^{1,0}\mathfrak{R}_I$, the $m\times n$ matrix with $1$ at the $(1,2)$-position and zero elsewhere. Then
 $$(f_0)_\ast(V_0)=\begin{pmatrix}
         1 & 0 & \cdots & 0 & 0 & \cdots & 0 \\
         0 & 1 & \cdots & 0 & 0 & \cdots & 0 \\
         \vdots & \vdots & \cdots & 0 & \vdots & \cdots & 0 \\
         0 & 0 & \cdots & 1 & 0 & \cdots & 0 \\
       \end{pmatrix}_{m\times n}.$$
 A direct computation at $Z=0$ gives
 $$F_{t,k}(0;V_0)=\sqrt{m+n}$$
 while
 $$F_{t,k}(f_0(0);(f_0)_\ast(V_0))=\sqrt{\frac{m+n}{1+t}\cdot(m+t\sqrt[k]{m})}.$$
 Thus
$$(f_0^\ast F_{t,k})(0;V_0)=\sqrt{\frac{m+t\sqrt[k]{m}}{1+t}}F_{t,k}(0;V_0)=k_0(\mathfrak{R}_I,F_I)F_{t,k}(0;V_0).$$
However, $f_0$ is not surjective (its image consists of diagonal matrices with equal diagonal entries, a complex submanifold of dimension $1$), and its differential at $0$ has rank $1$ (since it only depends on $Z_{12}$). Hence $f_0\notin\mbox{Aut}(\mathfrak{R}_I)$.
\end{example}

The above two examples show that the equality of \eqref{thm-s1} in a single direction does not imply that $f$ is an automorphism of $\mathfrak{\mathfrak{D}}$. The single-direction equality is a sharpness condition for the Lu constant, but it is not sufficient for rigidity. Only the full tangent-space equality forces the map to be an automorphism.

\vskip0.4cm
\noindent
\textbf{Acknowledgement.}\ {\small The second  author was supported by the National Natural Science Foundation of China (Grant No. 12471080) and the Key Program of Fujian Provincial Natural Science Foundation of China (No. 2026J002007).}

\end{document}